\documentclass[11pt,a4paper,dvipsnames]{scrartcl}

\usepackage{geometry}
\usepackage{kbordermatrix}

\usepackage[sortcites, backend=biber, giveninits=true, style=ieee, dashed=false, sorting = nty,citestyle=numeric]{biblatex}

\AtEveryBibitem{\clearfield{month}
\clearfield{day}
\clearfield{doi}
}

\usepackage{amsmath,amsthm,amssymb,url}
\usepackage{csquotes}
\usepackage[shortlabels]{enumitem}

\setlist[enumerate]{itemsep=0mm,parsep=2mm,topsep=5pt,label=\textit{(\alph*)}}

\newtheorem{theorem}{Theorem}[section]
\newtheorem{proposition}[theorem]{Proposition}
\newtheorem{lemma}[theorem]{Lemma}
\newtheorem{claim}[theorem]{Claim}
\newtheorem{corollary}[theorem]{Corollary}

\newtheorem{conjecture}[theorem]{Conjecture}

\usepackage[english]{babel}
\usepackage{microtype}
\usepackage{mathtools}

\usepackage{pgf,tikz,pgfplots}
\pgfplotsset{compat=1.18}
\usepackage{subcaption}
\def\normaledge{1.2}
\definecolor{edgeblack}{rgb}{0.25,0.25,0.25}
\definecolor{vertexblack}{rgb}{0,0,0}

\usepackage{xcolor}
\colorlet{myurlcolor}{Aquamarine}
\definecolor{mycitecolor}{HTML}{5590B4}

\usepackage{hyperref}
\hypersetup{
linkcolor  = black,
citecolor  = mycitecolor!85!black, %
urlcolor   = myurlcolor!85!black,
colorlinks = true,
}
\usepackage[capitalise, noabbrev, nameinlink, nosort]{cleveref}
\crefname{equation}{}{}
\crefname{conjecture}{Conjecture}{Conjectures}
\crefname{claim}{Claim}{Claims}

\AddToHook{env/lemma/begin}{\crefalias{theorem}{lemma}}
\AddToHook{env/conjecture/begin}{\crefalias{theorem}{conjecture}}
\AddToHook{env/proposition/begin}{\crefalias{theorem}{proposition}}
\AddToHook{env/claim/begin}{\crefalias{theorem}{claim}}
\AddToHook{env/corollary/begin}{\crefalias{theorem}{corollary}}

\newcommand{\real}{{\mathbb{R}}}
\newcommand{\complex}{{\mathbb{C}}}
\newcommand{\rat}{{\mathbb{Q}}}

\newcommand{\R}{{\mathbb{R}}}

\newcommand{\cB}{{\mathcal{B}}}
\newcommand{\cR}{{\mathcal{R}}}
\newcommand{\cS}{{\mathcal{S}}}
\newcommand{\cH}{{\mathcal{H}}}
\newcommand{\cM}{{\mathcal{M}}}

\newcommand{\E}{\mathbb{E}}
\newcommand{\Prob}{\mathbb{P}}
\newcommand{\symcompletion}[1]{\mathcal{S}_{#1}}
\newcommand{\birigid}[1]{$#1$-birigid}

\newcommand{\locallycompletable}[1]{$d$-completable}

\DeclareMathOperator{\rank}{rank}

\begin{document}

\title{Sufficient conditions for the bipartite rigidity, symmetric completability and hyperconnectivity of graphs}

\author{D\'aniel Garamv\"olgyi\thanks{
HUN-REN-ELTE Egerv\'ary Research Group
on Combinatorial Optimization, P\'azm\'any P\'eter s\'et\'any 1/C, 1117 Budapest, Hungary.
e-mail: \texttt{daniel.garamvolgyi@ttk.elte.hu}}
\and
Bill Jackson\thanks{Queen Mary University of London, London, E1 4NS, UK, and
the HUN-REN-ELTE Egerv\'ary Research Group
on Combinatorial Optimization, P\'azm\'any P\'eter s\'et\'any 1/C, 1117 Budapest, Hungary.
e-mail: \texttt{b.jackson@qmul.ac.uk}}
\and
Tibor Jord\'an\thanks{Department of Operations Research, ELTE E\"otv\"os Lor\'and University, and the HUN-REN-ELTE Egerv\'ary Research Group
on Combinatorial Optimization, P\'azm\'any P\'eter s\'et\'any 1/C, 1117 Budapest, Hungary.
e-mail: \texttt{tibor.jordan@ttk.elte.hu}}
\and 
Soma Vill\'anyi\thanks{Department of Operations Research, ELTE E\"otv\"os Lor\'and University, and the HUN-REN-ELTE Egerv\'ary Research Group
on Combinatorial Optimization, P\'azm\'any P\'eter s\'et\'any 1/C, 1117 Budapest, Hungary. e-mail: \texttt{soma.villanyi@ttk.elte.hu}}}

\date{}

\maketitle
\begin{abstract}
    We consider three matroids defined by Kalai in 1985: the \emph{symmetric completion matroid} $\mathcal{S}_d$ on the edge set of a looped complete graph; the \emph{hyperconnectivity matroid} $\mathcal{H}_d$ on the edge set of a complete graph; and the \emph{birigidity matroid} $\mathcal{B}_d$ on the edge set of a complete bipartite graph. These matroids arise in the study of low rank completion of partially filled symmetric, skew-symmetric and rectangular matrices, respectively. We give sufficient  conditions for a graph $G$ to have maximum possible rank in these matroids. For $\mathcal{S}_d$ and $\mathcal{H}_d$, our conditions are in terms of the minimum degree of $G$ and are best possible. For $\mathcal{B}_d$, our condition is in terms of the connectivity of $G$. 
    
    Our results have several implications for the unique completability of low-rank matrices. In particular, they imply that: almost all sufficiently large $n \times n$ positive semidefinite matrices of rank $d$ are uniquely determined by any subset of their entries which includes at least $(n + d + 1)/2$ entries from each row; almost all $m \times n$  matrices of rank $d$ are uniquely determined by any subset of their entries whose positions define a spanning subgraph of $K_{m,n}$ which  is $k_d$-connected, for some constant $k_d=\mbox{O}(d^3)$.
\end{abstract}

\section{Introduction}

We consider three families of matroids defined by Kalai~\cite{K} on the edge set of a graph $G=(V,E)$.  Suppose $d\geq 1$ is an integer and  $p:V\to \R^d$ is  a realisation of $G$ in $\R^d$. We say that $p$ is {\em generic} if the multiset of coordinates of the points $p(v)$, $v\in V$, is algebraically independent over $\rat$.
\begin{itemize}
    \item When $G$ is {\em semisimple}, i.e.~each vertex of $G$ is incident with at most one loop and no  parallel edges, the {\em symmetric completion matroid of $(G,p)$}, denoted by $\cS_d(G,p)$, is the row matroid of the $|E|\times d|V|$ matrix $S(G,p)$ with rows indexed by $E$ and sets of $d$ consecutive columns indexed by $V$, in which the row indexed by a non-loop edge $uv\in E$ is 
    \vspace{-.5em}
\[
    \kbordermatrix{
    & &  u & & v & \\
    e=uv & 0 \dots 0 & p(v) & 0\dots 0 & p(u) & 0\dots 0
    }
\]
    and the row indexed by a loop edge $uu\in E$ is
    \vspace{-.5em}
\[
    \kbordermatrix{
    & &  u  & \\
    e=uu & 0 \dots 0 & p(u)  & 0\dots 0
    }.
\]
    The {\em $d$-dimensional  symmetric completion matroid of $G$},  denoted by $\cS_d(G)$, is given by the matroid $\cS_d(G,p)$ for any generic $p$. Note that this is well-defined, i.e.\ it does not depend on the (generic) choice of $p$.
    
    \item When $G$ is simple, the {\em hyperconnectivity  matroid} of $(G,p)$,  denoted by $\cH_d(G,p)$, is the row matroid of the $|E|\times d|V|$ matrix $H(G,p)$ with rows indexed by $E$ and sets of $d$ consecutive columns indexed by (a fixed ordering of the vertices in) $V$, in which the row indexed by an edge $uv\in E$ with $u<v$ is
    \vspace{-.5em}
\[
    \kbordermatrix{
    & &  u & & v & \\
    e=uv & 0 \dots 0 & p(v) & 0\dots 0 & -p(u) & 0\dots 0
    }.
\]
    The {\em $d$-dimensional  hyperconnectivity matroid of $G$},  denoted by $\cH_d(G)$, is given by the matroid $\cH_d(G,p)$ for any generic $p$. Note that $\cH_d(G,p)$ does not depend on the chosen ordering of $V$ or the (generic) choice of $p$.
    
    \item When $G$ is bipartite, the $d$-dimensional  symmetric completion  and hyperconnectivity  matroids of $G$ are identical. We refer to this common matroid as the {\em  $d$-dimensional birigidity  matroid} of $G$, and denote it by $\cB_d(G)$.
\end{itemize}

Each of these matroids appears in the study of low rank matrix completion problems. For example, a partially  filled $m\times n$  matrix $M$ with generic entries is completable to a matrix of rank at most $d$ over $\complex$ if and only if the set of edges of the complete bipartite graph $K_{m,n}$ defined by the positions of the entries in $M$ is independent in $\cB_d(K_{m,n})$. Similar results link the rank $d$ symmetric matrix completion problem to $\cS_d$-independence, and the rank $2d$ skew-symmetric matrix completion problem 
to $\cH_{2d}$-independence.  
We refer the reader to~\cite{B,JJT, JJT2, RS, SC} for more information on these links. The links to skew-symmetric completion will be discussed in more detail in \cref{sec:skew} below.

All three matroids were characterised when $d=1$ by Kalai~\cite{K}: $\cS_1(G)$ is the even cycle matroid of a semisimple graph $G$;  $\cH_1(G)$ and $\cB_1(G)$ are both equal to the cycle matroid of $G$ when $G$ is simple, respectively bipartite. No polynomial algorithm for checking independence in these matroids is known for $d\geq 2$, although a graph theoretic NP-certificate for independence in $\cH_2(G)$  is given by Bernstein in~\cite{B}. 
We do at least know the maximum possible rank of each of these matroids. Let $K_n^\circ, K_n, K_{m,n}$ denote the complete semisimple graph on $n$ vertices, the complete simple graph on $n$ vertices, and the complete bipartite graph in which the sets of the bipartition have cardinality $m$ and $n$. Then Kalai~\cite{K} gives the following.
\begin{lemma}\label{lem:kalai}
\begin{align*}
    \rank \cS_d(K_n^\circ)&=\begin{cases}
        dn-\binom{d}{2} \mbox{ if $n\geq d$},\\  
        \binom{n+1}{2}  \mbox{ if $n\leq d$};
    \end{cases}\\
    \rank \cH_d(K_n)&=\begin{cases}
        dn-\binom{d+1}{2} \mbox{ if $n\geq d$},\\  
        \binom{n}{2}  \mbox{ if $n\leq d$};
    \end{cases}\\
    \rank \cB_d(K_{m,n})&=\begin{cases}
        d(m+n)-d^2 \mbox{ if $n,m\geq d$},\\  
        nm \mbox{  if $\min\{n,m\}\leq d$}.
    \end{cases}
\end{align*}
\end{lemma}

We say that a semisimple graph $G\subseteq K_n^\circ$  is {\em \locallycompletable{d}} if its edge set spans $\cS_d(K_n^\circ)$, that   a simple graph $G\subseteq K_n$  is {\em $d$-hyperconnected} if its edge set spans $\cH_d(K_n)$, and that a bipartite graph $G\subseteq K_{m,n}$  is {\em $d$-birigid}  if its edge set spans $\cB_d(K_{m,n})$. In this paper we obtain sufficient conditions for a graph to have these properties.

The \emph{degree} of a vertex $v$ in a semisimple graph $G$ is defined as the number of edges incident to $v$, counting a loop only once. We denote the minimum degree of $G$ by $\delta(G)$.
Our first main result is the following theorem, whose second part confirms a conjecture of Jackson, Jordán and Tanigawa (\cite[Conjecture 38]{JJT2}). 

\begin{theorem}\label{thm:mindegree}
    For every integer $d\geq 1$, there exist integers $h_d = O(d^2)$ and $s_d = O(d^2)$ such that the following hold.
    \begin{enumerate}[(a)]
        \item Every simple graph  $G$ on $n \geq h_d$ vertices with $\delta(G)\geq (n+d - 1)/2$ is $d$-hyperconnected. 
        \item Every semisimple graph $G$ on $n \geq s_d$ vertices with the property that $\delta(G) \geq (n+d-1)/2$ and all vertices which are not incident with a loop have degree at least $ (n+d)/2$ is $d$-completable. 
    \end{enumerate}
\end{theorem}

\noindent
The degree bounds in \cref{thm:mindegree} are best possible when $d \geq 2$; see the discussion after the proof of \cref{thm:mindegree} in \cref{section:mindegree}. %

Our second main result shows that every sufficiently highly connected bipartite graph is $d$-birigid.

\begin{theorem}\label{thm:main} For every integer $d\geq 1$, there exists an integer $k_d=O(d^3)$ such that every $k_d$-connected bipartite graph is $d$-birigid.
\end{theorem}

\noindent
We  will construct examples of $d^2$-connected bipartite graphs which are not $d$-birigid in \cref{subsection:concludingbipartite}. Thus the connectivity condition in \cref{thm:main} is within a factor of $d$ of being best possible.

The reader may wonder whether every sufficiently highly connected graph is \locallycompletable{d} or $d$-hyperconnected. This is not true in general: for any bipartite graph $G$ on $n$ vertices and with vertex classes of size at least $d$ we have $\rank \cS_d(G) =\rank \cH_d(G)=\rank \cB_d(G)\leq dn - d^2$, and hence such a graph cannot be \locallycompletable{d} or $d$-hyperconnected.

\subsection{Unique realisability and matrix completion}

We can use \cref{thm:mindegree,thm:main} to deduce the following result about unique completability of low rank matrices. Let $M_d(m,n)$ and $S_d(n)$ denote the set of real $n \times m$ matrices of rank $d$, and the set of 
positive semidefinite
$n \times n$ matrices of rank $d$, respectively. Each $M\in  M_d(m,n)$  can be factored as $M=A^TB$ for some $d\times m$ matrix $A$ and some $d\times n$ matrix $B$, and we say that $M$ is {\em generic (in $M_d(m,n)$) } if the multiset containing all entries in $A$ and $B$ is algebraically independent over $\rat$ for some factorisation $M=A^TB$. 
Similarly, each $S\in  S_d(n)$  can be factored as $M=A^TA$ for some $d\times n$ matrix $A$, and we say that $S$ is {\em generic (in $S_d(n)$)} if the multiset containing all entries in $A$  is algebraically independent over $\rat$ for some factorisation $S=A^TA$.

Let $s_d,k_d$ be the constants defined in \cref{thm:mindegree,thm:main}, respectively.

\begin{theorem}\label{thm:matrix} Let $d\geq 1$ be an integer.
    \begin{enumerate}
        \item Every generic matrix in $S_d(n)$  is uniquely determined by any subset of its entries which includes at least $(n + d + 1)/2$ entries from each row.
        \item Every generic matrix in $M_d(m,n)$ is uniquely determined by any subset of its entries that defines a $(k_d+1)$-connected spanning subgraph of $K_{m,n}$.
    \end{enumerate}
\end{theorem}

To derive \cref{thm:matrix}, we consider the following ``global'' versions of $d$-completability and $d$-birigidity.
We say that a semisimple graph $G = (V,E)$ is \emph{globally $d$-completable} if, for every generic realisation $p:V \to \R^d$ and every realisation $q: V \to \R^d$ satisfying $p(u) \cdot p(v) = q(u) \cdot q(v)$ for all $uv \in E$, the equality $p(u) \cdot p(v) = q(u) \cdot q(v)$ holds for all $u,v \in V$. Then $G$ is globally $d$-completable if and only if every generic matrix in $S_d(|V|)$ is uniquely determined by its entries corresponding to the edges of $G$. This follows from the definition by labeling the vertices of $G$ as $V = \{v_1,\ldots,v_n\}$, and then associating to each generic matrix $M = A^TA \in S_d(|V|)$ the generic realisation $p$ of $G$ given by letting $p(v_i)$ be the $i$'th row of $A$.

Similarly, when $G$ is bipartite with bipartition $V=X\sqcup Y$, we say that $G$ is \emph{globally $d$-birigid} if, for every generic realisation $p:V \to \R^d$ and every realisation $q: V \to \R^d$ satisfying $p(u) \cdot p(v) = q(u) \cdot q(v)$ for all $uv \in E$, the equality $p(u) \cdot p(v) = q(u) \cdot q(v)$ holds for all $u \in X,v \in Y$. Then $G$ is globally $d$-birigid if and only if every generic matrix in $M_d(|X|,|Y|)$ is uniquely determined by its entries corresponding to the edges of $G$. This follows from the definition by labeling the vertices of $G$ as $X = \{x_1,\ldots,x_m\}$ and $Y = \{y_1,\ldots,y_n\}$, and then associating to each generic matrix $M = A^TB \in M_d(|X|,|Y|)$ the realisation $p$ given by letting $p(x_i)$ be the $i$'th row of $A$, and $p(y_j)$ be the $j$'th row of $B$. %

Further details on global $d$-completability, global $d$-birigidty  and their links to matrix completion can be found in \cite{SC,JJT2}.
In particular,
Singer and Cucuringu \cite{SC} show that $G$ is globally $1$-completable if and only if $G$ is a connected semisimple graph which contains at least one odd cycle, and $G$ is globally $1$-birigid if and only if $G$ is a connected bipartite graph. No characterisations are known for $d\geq 2$. 

The properties of being $d$-completable and $d$-birigid are necessary conditions for a graph to be globally $d$-completable and globally $d$-birigid, respectively. They also give rise to the following sufficient conditions.

\begin{lemma}\label{thm:redundant} Suppose $d\geq 2$ is an integer and $G=(V,E)$ is a graph.\begin{enumerate}
\item If $G-v$ is  $d$-completable for all $v\in V$, then $G$ is globally $d$-completable.
\item If $G$ is  bipartite and $G-v$ is $d$-birigid for all $v\in V$, then $G$ is globally $d$-birigid.
\end{enumerate}
\end{lemma}

\noindent
\cref{thm:redundant}(a) is given in \cite[Theorem 31]{JJT2}. \cref{thm:redundant}(b) is a new result which we prove in \cref{sec:red}.
The characterisations of global completability and birigidity when $d=1$, and a combination of \cref{thm:mindegree,thm:main,thm:redundant} when $d\geq 2$, imply the following result.

\begin{theorem} \label{thm:glob} Let $d\geq 1$ be an integer. 
\begin{enumerate}[(a)] \itemsep0pt
\item Every semisimple graph $G$ on $n \geq s_d+1$ vertices with the property that $\delta(G) \geq (n+d)/2$ and all vertices which are not incident with a loop have degree at least $ (n+d+1)/2$ is  globally $d$-completable.
\item Every $(k_d+1)$-connected bipartite graph  is globally $d$-birigid.
\end{enumerate}
\end{theorem}

\noindent
\cref{thm:glob}, in turn, immediately implies \cref{thm:matrix}.

\noindent

\subsection{Completability and rigidity}
We close this introductory section by describing how \cref{thm:mindegree,thm:main} relate to recent results on the $d$-dimensional rigidity matroid of a graph.
Given a %
simple graph $G=(V,E)$ and a map $p:V\to \R^d$,  
the {\em rigidity  matrix} of $(G,p)$ is 
the $|E|\times d|V|$ matrix $R(G,p)$ with rows indexed by $E$ and sets of $d$ consecutive columns indexed by $V$, in which the row indexed by an edge $uv\in E$ is 
\vspace{-.5em}
\[ 
\kbordermatrix{
& &  u & & v & \\
e=uv & 0 \dots 0 & p(u)-p(v) & 0\dots 0 & p(v)-p(u) & 0\dots 0
}.
\]
The {\em $d$-dimensional  rigidity  matroid} of $G$, denoted by $\cR_d(G)$, is the row matroid of $R(G,p)$ for any generic $p$.

A graph $G\subseteq K_n$ is said to be {\em $d$-rigid} if its edge set spans $\cR_d(K_n)$. By \cite[Corollary 2.6]{JJT}, a graph $G$ is $d$-rigid if and only if the semisimple graph $G^\circ$ obtained by adding a loop at each vertex of $G$ is $(d+1)$-completable. Thus any characterisation of $\cS_{d+1}(K_n^\circ)$ would give a characterisation of $\cR_{d}(K_n)$.

Krivelevich, Lew and Michaeli~\cite{KLM} and Vill\'anyi~\cite{vill} have recently used the probabilistic method to obtain analogous results to \cref{thm:mindegree,thm:main} for $d$-rigidity: Krivelevich et al.\ showed that every sufficiently large graph with minimum degree at least $(n+d-2)/2$ is $d$-rigid; and Vill\'anyi showed that every $d(d+1)$-connected graph is $d$-rigid. (Note that the result of Krivelevich et al.\ can also be deduced from \cref{thm:mindegree}(b) by using the above-mentioned link between $\cS_{d+1}(K_n^\circ)$  and $\cR_d(K_n)$.)
The key difference between our setting and that of \cite{KLM,vill} is that the so-called 1-extension operation preserves independence in $\cR_d(K_n)$, but does not preserve independence in the matroids we work with. Instead, we have to rely on the double 1-extension and looped 1-extension operations defined in the next section, and this requires significant new ideas. In particular, to prove \cref{thm:main}, we introduce a new variant of vertex-connectivity for bipartite graphs, which we call $k$-biconnectivity, and we give a lower bound on the vertex cover number in critically $k$-biconnected bipartite graphs (\cref{lem:tau_bound}), as well as in critically $k$-connected graphs (\cref{lem:tau_bound2}).

\section{Terminology and preliminary results}
Henceforth, we will assume that $d$ is a fixed positive integer.
We will use the following terminology throughout this paper for the four families of matroids $\cS_d(K_n^\circ)$, $\cH_d(K_n)$, $\cR_d(K_n)$ and $\cB_d(K_{m,n})$.
Let $G_0=(V_0,E_0)$ be a semisimple graph and $\cM$ be a matroid defined on $E_0$ with rank function $r$.
We say that a subgraph $G=(V,E)$ of $G_0$ is \emph{$\cM$-independent} if $r(E)=|E|$, and that a subgraph $G'=(V',E')$ of $G$ is an \emph{$\cM$-basis of $G$} if $G'$ is $\cM$-independent and $r(E') = r(E)$.  For vertices $u,v \in V$ with $uv\in E_0$ (possibly $u = v$), we say that $\{u,v\}$ is \emph{$\cM$-linked in $G$} if $r(E) = r(E+uv)$. The graph $G$ is {\em $\cM$-closed} if all $\cM$-linked vertex pairs in $G$ are edges of $G$. This is equivalent to saying that $E$ is a closed set in $\cM$.

\medskip
We will also use the following three operations defined on a semisimple graph $G=(V,E)$.
\begin{itemize}
    \item The \emph{($d$-dimensional) $0$-extension operation} constructs a new graph $H$ from $G$ by adding a new vertex $v$  and joining $v$ to $d$ vertices $v_1, \dots, v_d\in V+v$ (adding a loop $vv$ when $v\in \{v_1,v_2,\ldots,v_d\}$). We will refer to the special case of this operation that does not add a loop at $v$ as a {\em simple 0-extension}.
    \item The \emph{($d$-dimensional) double $1$-extension operation} on an edge $xy\in E$ constructs a new graph $H$ by adding two new vertices $u,v$  to $G-xy$, joining $u$ to a set of $d$ vertices in $V+u$ which includes $x$ (and may include $u$), joining $v$ to a set of $d$ vertices in $V+v$ which includes $y$ (and may include $v$), and finally adding the edge $uv$. We allow the possibility that $x=y$. We will refer to the special case of this operation that does not add a loop at $u$ or $v$ as a {\em simple double 1-extension}.
    \item The \emph{($d$-dimensional) looped $1$-extension operation}
    on a non-loop edge $xy\in E$ 
    constructs a new graph $H$ by adding a new vertex $v$ to $G-xy$, joining $v$ to a set of $d$ vertices in $V$ which includes $x$ and $y$, and adding the loop $vv$. In this operation $x=y$ is not allowed.
\end{itemize}

The first part of the following lemma is given in~\cite[Lemmas 2.3, 4.1]{JJT} and~\cite[Corollary 30]{JJT2}. The proof of the second part is similar, but we include it for completeness.

\begin{lemma}\label{lem:ext}
\begin{enumerate}[(a)] \itemsep0pt
    \item The $d$-dimensional $0$-extension, double $1$-extension and looped 1-extension operations preserve the property of being $\symcompletion{d}$-independent ($d$-completable, respectively).
    \item The $d$-dimensional simple $0$-extension and simple double $1$-extension operations preserve the property of being $\cH_{d}$-independent ($d$-hyperconnected, respectively).
\end{enumerate}
\end{lemma}
\begin{proof}[Proof of (b)] We first suppose that $G=(V,E)$ is a $\cH_d$-independent simple graph and that $G'=(V+v,E')$ is obtained from $G$ by a simple $0$-extension operation which adds a new vertex $v$ and new edges $vv_1,vv_2,\ldots,vv_d$. Let $p:V+v\to \R^d$ be a generic realisation of $G'$. Then
$H(G',p)=
    \begin{pmatrix}
        A&*\\
        0&B
    \end{pmatrix}$
    where  $A$ is the $d\times d$ matrix with rows $p(v_1),\ldots,p(v_d)$ and $B=H(G,p|_V)$. Hence $\rank H(G',p)=\rank A+\rank H(G,p|_V)=d+|E| =  |E'|$, so $G'$ is $\cH_{d}$-independent.
    
    We next suppose that $G=(V,E)$ is an $\cH_d$-independent simple graph and $G'=(V+v,E')$ is obtained from $G$ by a simple double $1$-extension operation which deletes an edge $xy$ and adds two new vertices $x',y'$ and the new edge $x'y'$; $d$ new edges from $x'$ to a set of $d$ vertices in $V$ including $y$; and $d$ new edges from $y'$ to a set of $d$ vertices in $V$ including $x$. Let $p:V\to \R^d$ be a generic realisation of $G$. Consider the (non-generic) realisation $p':V+x'+y'\to \R^d$ of $G'+xy$ obtained by putting $p'(x')=p(x)$, $p'(y')=p(y)$ and $p'(v)=p(v)$ for all $v\in V$. The graph $G'-x'y'$ can be obtained from $G$ by two simple 0-extensions and, since $p$ is generic, we may use the argument in the previous paragraph to deduce that $\rank H(G'+xy-x'y',p')=\rank H(G,p)+2d$. This implies that  $\rank H(G'+xy,p)\geq \rank H(G,p)+2d$. On the other hand, the rows of $H(G'+xy,p)$ indexed by the edges $xy,xy',x'y,x'y'$ have the form
\[
    \kbordermatrix{
    &  x &  x'&  y  & y' &  & & \\
    xy & p(y) & 0 & -p(x) & 0 & 0 &\dots &0 \\
    xy' & p(y) & 0 & 0 & -p(x) & 0 &\dots &0 \\ 
    x'y & 0 & p(y) & -p(x) & 0 & 0 &\dots &0 \\ 
    x'y' & 0 & p(y) & 0 & -p(x) & 0 &\dots &0 
    } \,.
\]
    It is straightforward to check that these rows are a circuit in the row matroid of $H(G'+xy,p')$ and hence $\rank H(G',p')=\rank H(G'+xy,p')\geq \rank H(G,p)+2d=|E'|$. This implies that the rows of $ H(G',p')$ are linearly independent. The same will be true for any generic realisation of $G'$ and hence $G'$ is  $\cH_d$-independent.

    That both operations preserve the property of being $d$-hyperconnected follows immediately from the fact that the latter is equivalent to the existence of an $\cH_d$-independent subgraph on $d|V| - \binom{d+1}{2}$ edges (when $|V| \geq d$).
\end{proof}

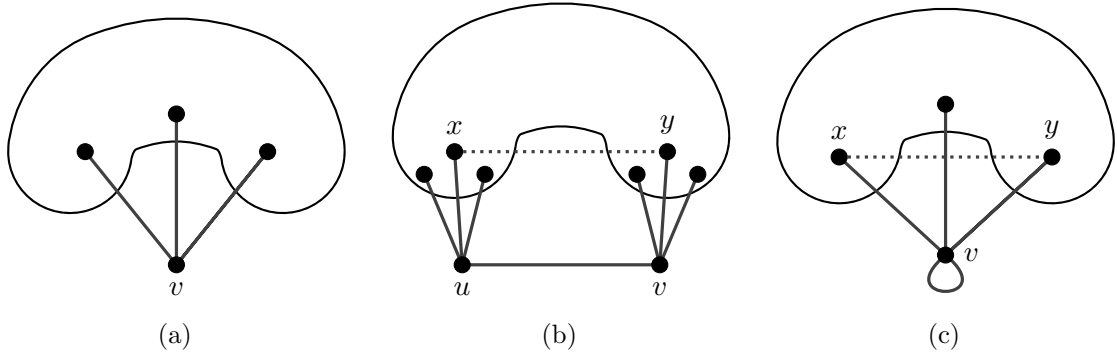
\begin{figure}[t]
    \centering
    \begin{subfigure}[b]{0.32\linewidth}
        \centering
        \begin{tikzpicture}[x = 1cm, y = 1cm, scale = 1]
            
            \node (b) at (-1.4,1){};
            \node (c) at (-0.8,1.8){};
            \node (d) at (0.8,1.8){};
            \node (e) at (1.4,1){};
            
            \node (v) at (0,-.5){};
            
            \node (x) at (-1.2,1){};
            \node (y) at (1.2,1){};
            \node (z) at (0,1.5){};
            
            \draw[line width=0.8,double distance=16mm,smooth,line cap=round,tension=.9] plot coordinates { (b) (c) (d) (e)};
            
            \draw [line width=\normaledge,color=edgeblack] (x.center) -- (v.center) -- (y.center) -- (v.center) -- (z.center);
            \draw [fill=vertexblack] (v) circle (3pt) node[below=3pt]{$v$};
            \draw [fill=vertexblack] (x) circle (3pt);
            \draw [fill=vertexblack] (y) circle (3pt);
            \draw [fill=vertexblack] (z) circle (3pt);

        \end{tikzpicture}
        \caption{}
    \end{subfigure}
    \begin{subfigure}[b]{0.32\linewidth}
        \centering
        \begin{tikzpicture}[x = 1cm, y = 1cm, scale = 1]
            
            \node (b) at (-1.4,1){};
            \node (c) at (-0.8,1.8){};
            \node (d) at (0.8,1.8){};
            \node (e) at (1.4,1){};

            \node (x) at (-1.4,.8){};
            \node (y) at (1.4,0.8){};
            
            \node (x') at (-1.3,-0.7){};
            \node (y') at (1.3,-0.7){};
            
            \draw[line width=0.8,double distance=16mm,smooth,line cap=round,tension=.9] plot coordinates { (b) (c) (d) (e)};
            
            \draw [line width=\normaledge, color=edgeblack, dotted] (x.center) to (y.center);
            \draw [line width=\normaledge,color=edgeblack] (x.center) -- (x'.center) -- (y'.center) -- (y.center);
            
            \node (a1) at (-1.8,.5){};
            \node (a2) at (-1,.5){};
            \node (b1) at (1.8,.5){};
            \node (b2) at (1,.5){};
            
            \draw [line width=\normaledge,color=edgeblack] (a1.center) -- (x'.center) -- (a2.center);
            \draw [line width=\normaledge,color=edgeblack] (b1.center) -- (y'.center) -- (b2.center);

            \draw [fill=vertexblack] (x) circle (3pt) node[above=3pt]{$x$};
            \draw [fill=vertexblack] (y) circle (3pt) node[above=3pt]{$y$};
            \draw [fill=vertexblack] (x') circle (3pt) node[below=3pt]{$u$};
            \draw [fill=vertexblack] (y') circle (3pt) node[below=3pt]{$v$};
            \draw [fill=vertexblack] (a1) circle (3pt);
            \draw [fill=vertexblack] (a2) circle (3pt);
            \draw [fill=vertexblack] (b1) circle (3pt);
            \draw [fill=vertexblack] (b2) circle (3pt);
            
        \end{tikzpicture}
        \caption{}
    \end{subfigure}
    \begin{subfigure}[b]{0.32\linewidth}
        \centering
        \begin{tikzpicture}[x = 1cm, y = 1cm, scale = 1]
            
            \node (b) at (-1.4,1){};
            \node (c) at (-0.8,1.8){};
            \node (d) at (0.8,1.8){};
            \node (e) at (1.4,1){};

            \node (x) at (-1.4,.8){};
            \node (y) at (1.4,0.8){};
            \node (z) at (0,1.5){};
            
            \node (x') at (-1.3,-0.7){};
            \node (y') at (1.3,-0.7){};

            \draw[line width=0.8,double distance=16mm,smooth,line cap=round,tension=.9] plot coordinates { (b) (c) (d) (e)};
            
            \draw [line width=\normaledge, color=edgeblack, dotted] (x.center) to (y.center);

            \draw [line width=\normaledge,color=edgeblack] (x.center) -- (v.center) -- (y.center) -- (v.center) -- (z.center);
            
            \draw [line width=\normaledge,color=edgeblack] (v.center)  to[min distance=10mm,in=-40,out=-140] (v.center);
            
            \draw [fill=vertexblack] (v) circle (3pt) node[right=3pt]{$v$};
            \draw [fill=vertexblack] (x) circle (3pt) node[above=3pt]{$x$};
            \draw [fill=vertexblack] (y) circle (3pt) node[above=3pt]{$y$};
            \draw [fill=vertexblack] (z) circle (3pt);
        \end{tikzpicture}
        \caption{}
    \end{subfigure}
    \caption{Examples of \emph{(a)} a (simple) 0-extension, \emph{(b)} a (simple) double 1-extension, and \emph{(c)} a looped 1-extension in the $d=3$ case.}\label{fig:operations}
\end{figure}

\section{Matroid seeds and their basic properties}

Throughout this section, we will assume that 
$G_0=(V_0,E_0)$ is a semisimple graph,  $\cM$ is a matroid on $E_0$ with rank function $r$ and $G=(V,E)$ is a subgraph of $G_0$. We will often denote $r(E)$ by $r(G)$ in order to simplify notation. We say that $\cM$ has the {\em $d$-dimensional 0-extension property} if for all subgraphs $G_1,G_2$ of $G_0$ such that $G_1$ is $\cM$-independent and $G_2$ is a $d$-dimensional 0-extension of $G_1$, the graph $G_2$ is $\cM$-independent.

    We next introduce the main technical tool that we will use in our proofs of \cref{thm:mindegree,thm:main}. 
A subset $K\subseteq V$ is an \emph{$\cM$-seed of $G$ (with respect to $d$)} if
\begin{itemize} \itemsep0pt
    \item $r(G)=r(G[K])+d|V - K|$, and
    \item for every $K\subseteq K'\subsetneq V$, there is a vertex $x\in V-K'$ such that $|(K'+x)\cap N_G(x)|\geq d$, where $N_G(x)$ denotes the set of vertices of $G$ which are adjacent to $x$.  
\end{itemize}
We have the following characterisation of $\cM$-seeds in the case when $\cM$ has the $d$-dimensional 0-extension property.
\begin{lemma}\label{lem:seed:altdef}
    Suppose that $\cM$ has
    the $d$-dimensional $0$-extension property. 
    \begin{enumerate}[(a)]
        \item A subset $K \subseteq V$ is an $\cM$-seed of $G$ if and only if there is an $\cM$-independent subgraph $I_K = (K,F)$ of $G[K]$ and an $\cM$-basis $B_G$ of $G$ such that $B_G$ can be obtained from $I_K$ by a series of $d$-dimensional $0$-extensions.
        \item If $K \subseteq V$ is an $\cM$-seed of $G$ and $v\in V-K$ satisfies $|(K+v)\cap N_G(v)|\geq d$, then $K+v$ is also an $\cM$-seed.
    \end{enumerate}
\end{lemma}
\begin{proof} \emph{(a)}
    To prove necessity, let us suppose that $K$ is an $\cM$-seed of $G$. Choose an $\cM$-basis $B_K$ of $G[K]$.  Since $K$ is an $\cM$-seed, there is an ordering $v_1,\ldots,v_k$ of $V-K$ such that 
    $|(K \cup \{v_1,\ldots,v_{i}\}) \cap N_G(v_i)| \geq d$,
    for all $i \in \{1,\ldots,k\}$. Since $\cM$ has the $0$-extension property, we can construct an $\cM$-independent subgraph $B_G$ of $G$ from $B_K$ by a series of 0-extensions using $v_1,v_2,\ldots,v_k$. Since \[r(B_G) = |E(B_G)| = |E(B_K)| + dk = r(G[K]) + d|V - K| = r(G),\] $B_G$ is an $\cM$-basis of $G$, as desired.
    
    To prove sufficiency, we suppose that an $\cM$-basis $B_G$ of $G$ can be obtained from an $\cM$-independent subgraph $I_K$ of $G[K]$ by a series of 0-extensions. Let $v_1,\ldots,v_k$ be the ordering of $V-K$ along which we perform these 0-extensions. We have 
    \[r(G) = |E(B_G)| = |E(I_K)| + d|V-K| \leq r(G[K]) + d|V-K| \leq r(G),\] 
    where the last inequality follows from 
    the hypothesis that $\cM$ has the 0-extension property. Hence $r(G)=r(G[K])+d|V - K|$. In addition, if $K \subseteq K' \subsetneq V$, then the construction of $B_G$ implies that $|(K'+v_i)\cap N_G(v_i)|\geq d$, where $i \in \{1,\ldots,k\}$ is the smallest index such that $v_i \notin K'$. This implies that $K$ is an $\cM$-seed and completes the proof of \emph{(a)}.

    \emph{(b)} This follows immediately from part \emph{(a)}.
\end{proof}

We prove three more lemmas in this section. The first,
\cref{lem:small_kernel}, provides an upper bound on the size of the smallest $\cM$-seed of a semisimple graph $G=(V,E)$. We will use it to obtain a sufficient condition for $G$ to have 
a (small) $\cM$-seed $K$. %
The second, \cref{lemma:deletablevertices}, shows that if $G$ has an $\cM$-seed 
which does not cover the edges of $G$,
then we can find a pair of vertices $u,v$ whose deletion only decreases the rank of $G$ by a small amount. The third, \cref{lem:neighbours:clique}, guarantees that the neighbour sets of such $u$ and $v$ induce a dense subgraph of $G$ when $G$ is $\cM$-closed and $\cM$ satisfies certain additional properties. The existence of these dense subgraphs will form the basis of our arguments in later sections.

\begin{lemma}\label{lem:small_kernel}
    Suppose that $\cM$ has
    the $d$-dimensional $0$-extension property and has
    rank at most $dn$ for some $d\geq 2$. Let $t\geq 0$ be an integer and let $X_0\subseteq X_1\subseteq\dots\subseteq X_t=V$ be such that, for all $1\leq i\leq t$ and $v\in X_i-X_{i-1}$, we have $|N_G(v)\cap X_{i-1}|\geq d$.
    Then $G$ has an $\cM$-seed $K$ with \[|K|\leq 2|X_0|\frac{d^{t+1}}{d-1}.\]
\end{lemma}
\begin{proof}
    Let $G'=(V,E')$ be a subgraph of $G$ with $|E'|=d|V-X_0|$ edges such that for all $1\leq i\leq t$, $E'$ joins each $v\in X_i-X_{i-1}$ to $d$ vertices in $X_{i-1}$.  
    Since $\cM$ has the $d$-dimensional 0-extension property, $G'$ is $\cM$-independent. Let $E''\subseteq E-E'$ be a set of edges such that $B=(V,E'\cup E'')$ is an $\cM$-basis of $G$. Then $|E'\cup E''|\leq dn $ and hence
    \begin{equation*}\label{eq:2}
        |E''|\leq dn-|E'|= d|X_0|.
    \end{equation*}
    For every $i \in \{1,\ldots,t\}$ and $v\in X_i-X_{i-1}$, we iteratively define a set $Y_v$ as follows. If $v \in X_1 - X_0$, let $Y_v=\{v\}$.
    For $i=2,\dots,t$ and $v\in X_i-X_{i-1}$, let $$Y_v=\{v\}\cup\bigg(\bigcup_{u \in  N_{G'}(v)\cap X_{i-1}}Y_u\bigg).$$ 
    Let $Z$ denote $V(E'')-X_0$ and
    put $$K=X_0\,\cup\, \bigg(\bigcup_{v\in Z}Y_v\bigg).$$  Then $B$ can be obtained from $F=G'[K]+E''$ by a series of 0-extensions, first adding each vertex $v\in X_1-K$, one by one, followed by the vertices $v\in X_2-X_1-K$, and so on. Thus $K$ is an $\cM$-seed of $G$ by \cref{lem:seed:altdef}.
    Furthermore, $|Z|\leq 2|E''|\leq 2d|X_0|$, and $|Y_v|\leq 1+d+\dots +d^{t-1}=(d^{t}-1)/(d-1) $ for every $v\in V-X_0$. 
    Hence,
\[|K|\leq |X_0|+\sum_{v\in Z}|Y_v|\leq 2|X_0|\frac{d^{t+1}}{d-1},\]
    which completes the proof of the lemma. 
\end{proof}

\begin{lemma}\label{lemma:deletablevertices}
    Suppose that $\cM$ has the $d$-dimensional $0$-extension property. Suppose further that $G$ has minimum degree $\delta(G) \geq d+2$, and that $G$ has an $\cM$-seed $K \subseteq V$ and vertices $u',v' \in V-K$ with $u'v' \in E$ and $u' \neq v'$. Then there exist vertices $u,v \in V-K$ with $uv \in E$ and $u \neq v$ satisfying \[r(G) = r(G-u) + d = r(G-v)+d = r(G-u-v) + 2d.\]
\end{lemma}
\begin{proof}
    Let $K'$ be an $\cM$-seed of $G$ such that $K \subseteq K'$, there exist distinct vertices $u,v \in V-K$ with $uv \in E$, and $K'$ is maximal subject to these conditions. 
    Since $K'$ is an $\cM$-seed of $G$, there exists $w\in V-K'$ such that $|(K'+w)\cap N_G(w)|\geq d$. It follows from \cref{lem:seed:altdef} that $K' + w$ is also an $\cM$-seed of $G$, and hence by the maximality of $K'$, $K' + w$ covers every edge of $G$ that is not a loop. Hence each non-loop edge of $G$ not covered by $K'$ is incident to $w$. Thus $w\in \{u,v\}$ and we may assume, by symmetry, that $u=w$. Then all non-loop edges incident to $v$ in $G$, except for $uv$, are covered by $K'$, and in particular, $|(K'+v)\cap N_G(v)|\geq d$. It follows that $uv$ is the only non-loop edge in $G$ not covered by $K'$, for otherwise $K'+v$ would contradict the maximality of $K'$. Hence every vertex $w' \in V - K$ satisfies $|K' \cap N_G(w')| \geq d$, and thus by the maximality of $K'$, we have $K' = V - \{u,v\}$. \cref{lem:seed:altdef} now implies that $K'+u$ and $K'+v$ are also $\cM$-seeds of $G$, and hence \[r(G) = r(G-u) + d = r(G-v)+d = r(G-u-v) + 2d,\] as required.
\end{proof}

We say that $\cM$ has the {\em $d$-dimensional double 1-extension property} if, for all subgraphs $G_1,G_2$ of $G_0$ such that $G_1$ is $\cM$-independent and $G_2$ is obtained by a  $d$-dimensional double 1-extension of $G_1$, $G_2$ is $\cM$-independent. The {\em $d$-dimensional looped 1-extension property} is defined analogously.

\begin{lemma}\label{lem:neighbours:clique}
    Suppose that  $\delta(G) \geq d+1$.
    \begin{enumerate}[(a)]
        \item Suppose $\cM$ has the $d$-dimensional double 1-extension property. Let $uv \in E$ with $u \neq v$, and suppose that $r(G) = r(G-u-v)+2d$ holds. Then for every edge $xy \in E_0$ with $x \in N_G(u) - \{u,v\}$ and $y \in N_G(v) - \{u,v\}$, $\{x,y\}$ is $\cM$-linked in $G$.
        \item Suppose $\cM$ has the $d$-dimensional looped 1-extension property. Let $v \in V$ with $vv \in E$, and suppose that $r(G) = r(G-v)+d$. Then for every edge $xy \in E_0$ with $x,y \in N_G(v) - \{v\}$ and $x \neq y$, $\{x,y\}$ is $\cM$-linked in $G$. 
    \end{enumerate}
\end{lemma}
\begin{proof}
    \emph{(a)} Suppose, for a contradiction, that $\{x,y\}$ is not $\cM$-linked in $G$ for some $xy\in E_0$ with $x \in N_G(u) - \{u,v\}$ and $y \in N_G(v) - \{u,v\}$. Then $r(G-u-v+xy)=r(G-u-v)+1$. We can construct a subgraph of $G$ by applying a double 1-extension operation to $G-u-v+xy$. Since the double 1-extension operation preserves independence in $\cM$, we have 
    \[r(G)\geq r(G-u-v+xy)+2d=r(G-u-v)+2d+1,\] a contradiction.
    
    \emph{(b)} We proceed as in the previous case. Suppose, for a contradiction, that $\{x,y\}$ is not $\cM$-linked in $G$ for some $xy \in E_0$ with $x,y \in N_G(v) - \{v\}$ and $x \neq y$. Then $r(G-v+xy)=r(G-v)+1$. Since the looped 1-extension operation preserves independence in $\cM$, and we can construct a subgraph of $G$ by applying this operation to $G-v+xy$, we have 
    \[r(G)\geq r(G-v+xy)+d=r(G-v)+d+1,\] a contradiction.
\end{proof}

\section{The proof of \texorpdfstring{\cref{thm:mindegree}}{Theorem 1.2}}\label{section:mindegree}

We continue in the setting of the previous section. Let $G_0=(V,E_0)$ be a semisimple graph with $n$ vertices, and let $\cM=(E_0,r)$ be a matroid on $E_0$ with the $d$-dimensional 0-extension, double 1-extension and looped 1-extension properties. We say that a spanning subgraph $G=(V,E)$ of $G_0$ is {\em $\cM$-rigid} if $r(E)=r(E_0)$. 
Thus, $G$ is $d$-completable (resp.\ $d$-hyperconnected) if it is $\cM$-rigid for $\cM = \cS_d(K_n^\circ)$ (resp.\ $\cM = \cH_d(K_n)$).

Suppose that $\cM = \cS_d(K_n^\circ)$ or $\cM = \cH_d(K_n)$ and that $G$ has sufficiently large minimum degree.
We will prove that, under these conditions, if $G$ has an $\cM$-seed that does not cover the non-loop edges in $E$, then $G$ is $\cM$-rigid. \cref{lem:existsk:mindegree} below establishes the existence of such a seed by showing that $G$ has an $\cM$-seed of cardinality at most $n/3$.
In its proof, we will make use of the Chernoff bound for binomial random variables, see, e.g., \cite[Theorem 2.1]{janson2011random}.
\begin{lemma}\label{chernoff}
    Let $X\sim {\mathrm{Bin}}(k,p)$ and $0<\alpha<1$. Then the following hold.
    \begin{enumerate}[(a)]
        \item $\Prob\big(X\leq (1-\alpha) kp\big)\leq \exp \big(-\frac{\alpha^2 kp}{2}\big)$
        \item $\Prob\big(X\geq (1+\alpha) kp\big)\leq \exp \big(-\frac{\alpha^2 kp}{2+\alpha}\big)$
    \end{enumerate}
    
\end{lemma}

\begin{lemma}\label{lem:existsk:mindegree}
    Let $G_0$, $\cM$ and $G$ be as defined at the beginning of this section. Suppose that $d\geq2$, $n\geq 10^5 d^2$, and  $\delta(G)\geq (n+d-1)/2$. Then $G$ has an $\cM$-seed $K$ with $|K|< n/3$.
\end{lemma}
\begin{proof}
     We claim that 
     there exists a set $X_0\subseteq V$ of size $|X_0|< n/(12d)$ such that $|N_G(v)\cap X_0|\geq d$ for all $v\in V$.    Applying
     \cref{lem:small_kernel} to this $X_0$ with  $t=1$ and $X_1=V$ then gives an $\cM$-seed $K$ of $G$ with
    \[ |K|\leq \frac{2d^2}{d-1}|X_0|<  \frac{n}{3}, \]
    as required.
    
    To prove our claim, let $p=1/(16d)$, and let $X$ be a random subset of $V$ obtained by putting each vertex $v\in V$ into $X$, independently,  with probability $p$. 
    Then, for every $v\in V$, we have $|N_G(v)\cap X|\sim {\mathrm{Bin}}(\deg(v),p)$. Hence, by using \cref{chernoff}(a) with $\alpha = 1/2$, we obtain
    $$ \Prob\Big(|N_G(v)\cap X|<d\Big)\leq \Prob\Big(|N_G(v)\cap X|<\frac{\deg(v)\cdot p}{2}\Big)\leq \exp\Big(-\frac{\deg(v)\cdot p}{8}\Big).$$
    Let $A$ denote the event that there exists some $v\in V$ satisfying $|N_G(v)\cap X|<d$. Since $\deg(v)\cdot p\geq n/(32d)$, the union bound gives \[\Prob(A)\leq n\cdot\exp\left(-\frac{n}{256d}\right)\leq n\cdot\exp\left(-\sqrt n\right) <1/2.\] 
    A similar computation, using \cref{chernoff}(b) with $\alpha = 1/3$, shows that the event $B$ that $|X|\geq n/(12d)$ has probability $\Prob(B)<1/2$. Thus, there is a nonzero probability that neither $A$ nor $B$ occurs. This implies that there exists a set $X_0$ of size $|X_0|< n/(12d)$ satisfying $|N_G(v)\cap X_0|\geq d$ for all $v\in V$. 
\end{proof}

\paragraph*{Proof of \texorpdfstring{\cref{thm:mindegree}}{Theorem 1.1}.} \hspace{-1em}
Let us start by recalling that in the $d=1$ case, $\mathcal{H}_d(G)$ and $\mathcal{S}_d(G)$ are equal to the graphic matroid and the even cycle matroid of $G$, respectively. Thus $G$ is $1$-hyperconnected if and only if it is connected, and $G$ is $1$-completable if and only if each connected component of $G$ is non-bipartite. It is easy to verify that these conditions are satisfied under our assumptions on the minimum degree of $G$. Therefore, throughout the rest of the proof we may assume that $d \geq 2$.

We first prove \emph{(a)} by showing that $h_d=10^5d^2$ suffices.
To this end, assume that $n \geq 10^5d^2$, let $G = (V,E)$ be a simple graph on $n$ vertices, and let $\cM = \cH_d(K_n)$. Our goal is to show that $G$ is $\cM$-rigid; since this holds if and only if the $\cM$-closure of $G$ is $\cM$-rigid, we may assume that $G$ is $\cM$-closed. 

By \cref{lem:ext}(b), $\cM$ has the $d$-dimensional $0$-extension and double $1$-extension properties. It also vacuously has the $d$-dimensional looped $1$-extension property, since $K_n$ is loopless. Thus 
\cref{lem:existsk:mindegree}
implies that $G$ has an $\cM$-seed $K$ with $|K| < n/3$. Since $\delta(G)\geq (n+d - 1)/2 \geq |K|+2$, there exists an edge $u'v'\in E$ with $u',v'\in V-K$. \cref{lemma:deletablevertices} now implies that there exists an edge $uv \in E$ with $u,v \in V-K$ such that $r(G) = r(G-u-v) + 2d$.

Let $X = N_G(u) \cap N_G(v)$, and note that we have
\[|X| = |N_G(u)| + |N_G(v)| - |N_G(u) \cup N_G(v)| \geq n+d - 1 - n = d - 1.\]
\cref{lem:neighbours:clique}(a) and the fact that $G$ is $\cM$-closed imply that, for every $y \in N_G(v)$ and $x \in X$, we have $xy \in E$. In particular, $X$ is a clique of $G$, and hence so is $X+v$. Since $|X+v| \geq d$, we can obtain a spanning subgraph of $G[N_G(v)]$ from $G[X+v]$ by a series of $d$-dimensional $0$-extensions. A similar count shows that for every $w \in V - N_G(v)$, we have $|N_G(w) \cap N_G(v)| \geq d$, and hence we can obtain a spanning subgraph of $G$ from $G[N_G(v)]$ by a series of $0$-extensions. Since $G[X+v]$ is $\cM$-rigid, \cref{lem:ext}(b) now implies that $G[N_G(v)]$ and $G$ are both $\cM$-rigid, as desired.

The proof of \emph{(b)} is similar, but more involved. We show that we can take $s_d = 10^5d^2$.
Fix $n \geq 10^5d^2$, let $G = (V,E)$ be a semisimple graph on $n$ vertices, and let $\cM = \cS_d(K_n^\circ)$. Our goal is to show that $G$ is $\cM$-rigid, and hence we may assume, without loss of generality, that $G$ is $\cM$-closed. 

We first make a general observation. For each $v\in V$, let $\lambda_v=1$ if $v$ is incident with a loop in $G$, and otherwise put $\lambda_v=0$. By assumption, the degree of every vertex $v$ in $G$ is at least $(n+d-\lambda_v)/2$, and hence for any $u,v \in V$ with $u \neq v$, we have
\begin{equation}\label{eq:neighboursets}
    |N_G(u) \cap N_G(v)| = |N_G(u)| + |N_G(v)| - |N_G(u) \cup N_G(v)| \geq (n+d) - \frac{\lambda_u + \lambda_v}{2} - n.
\end{equation}
Hence $|N_G(u) \cap N_G(v)| \geq d-1$ with equality only if $u$ and $v$ are both incident with loops in $G$.

As in part \emph{(a)}, we can use \cref{lem:existsk:mindegree} 
and the hypothesis on $\delta(G)$ to find an $\cM$-seed $K$ of $G$ and an edge $u'v' \in E$ with $u',v' \in V-K$ and $u' \neq v'$. \cref{lemma:deletablevertices} now guarantees the existence of an edge $uv \in E$ with $u,v \in V-K$ and $u \neq v$ that satisfies \[r(G) = r(G-u) + d = r(G-v) + d = r(G-u-v) + 2d.\]

We claim that $G[N_G(v)]$ is $\cM$-rigid. Let $X = N_G(u) \cap N_G(v)$. \cref{lem:neighbours:clique}(a) and the assumption that $G$ is $\cM$-closed implies that $G[X]$ is a looped clique, and in particular, it is $\cM$-rigid. If $vv \notin E$, then \cref{eq:neighboursets} shows that $|X| \geq d$ and by \cref{lem:neighbours:clique}(a), $xy \in E$ for every $x \in X$ and $y \in N_G(v)$. It follows that we can obtain a spanning subgraph of $G[N_G(v)]$ from $G[X]$ using a series of $0$-extensions, and hence $G[N_G(v)]$ is $\cM$-rigid by \cref{lem:ext}(a). On the other hand, if $vv \in E$, then by \cref{lem:neighbours:clique}(b), $G[N_G(v)]$ is a (not necessarily looped) clique, and thus $u \in N_G(v)$ implies that $N_G(v) - u \subseteq X$. Hence $|X| \geq d$, and we may obtain a spanning subgraph of $G[N_G(v)]$ from $G[X]$ by adding $u$ using a 0-extension if $u\not\in X$. Once again, \cref{lem:ext}(a) implies that $G[N_G(v)]$ is $\cM$-rigid, as claimed.

To finish the proof, note that by \cref{eq:neighboursets}, for every $w\in V- N_G(v)$, 
either $ww\in E$ and $|N_G(w)\cap N_G(v)|\geq d-1$; or $ww \notin E$ and $|N_G(w)\cap N_G(v)|\geq d$. Thus we can obtain a spanning subgraph of $G$ from $G[N_G(v)]$ by a series of 0-extensions, and hence $G$ is $\cM$-rigid by \cref{lem:ext}(a). 
\qed

\medskip

It follows from \cref{lem:kalai} that the complete bipartite graph $K_{m,m}$ is neither $2$-hyperconnected nor 1-completable, for all $m\geq 2$. We can now use the fact that the so-called coning operation transforms a graph which is not $(d-1)$-hyperconnected to one which is not $d$-hyperconnected (see \cite[Theorem 5.1]{K}) to deduce that the complete tripartite graph  $K_{m,m,d-2}$ is not $d$-hyperconnected, for all $d\geq 2$. A similar argument, with \cite[Lemma 6]{JJT2} in place of \cite[Theorem 5.1]{K}, shows that $K_{m,m,d-1}$ is not  \locallycompletable{d}, for all $d\geq 1$. (See \cite[Lemma 9]{JJT2} for an explicit proof of this fact.) 
These examples show that the bound on $\delta(G)$ in \cref{thm:mindegree}(a) is tight when $d\geq 2$, and the bound on $\delta(G)$ in \cref{thm:mindegree}(b) is tight when $d\geq 1$. The bound in \cref{thm:mindegree}(a) is not tight when $d=1$, since $\delta(G)\geq (n-1)/2$ is sufficient to imply that $G$ is connected, and hence $1$-hyperconnected. 

We believe that the bounds $s_d=O(d^2)$ and $h_d=O(d^2)$ in the statement of \cref{thm:mindegree} are not optimal.
The theorem might remain true with $s_d=2d+2$; this would be best possible, since a $d$-completable graph on at least $d+1$ vertices must have at least $dn - \binom{d}{2}$ edges, and this is not guaranteed by the bound $\delta(G) \geq (n+d-1)/2$ in the regime $d+1 \leq n \leq 2d+1$. Similar considerations show that the optimal value for $h_d$ might be $2d$.

\section{The proof of \texorpdfstring{\cref{thm:main}}{Theorem 1.3}}

The inductive hypothesis in our proof of \cref{thm:main} requires us to work with a new 
version of connectivity for bipartite graphs. 
We say that a bipartite graph $G=(V,E)$
with bipartition $(A,B)$ 
is \emph{$k$-biconnected} if $|A|,|B|\geq k$ and for every subset $W$ of $V$ with $|W\cap A|\leq k-1$ and $|W\cap B|\leq k-1$, the graph $G-W$ is connected. Note that every $(2k-1)$-connected bipartite graph is $k$-biconnected. The graph $G$ is said to be \emph{critically $k$-biconnected} if it is $k$-biconnected, but for each $v\in V$, $G-v$ is not $k$-biconnected.

Given a graph $G=(V,E)$, we say that $A \subseteq V$ is a \emph{vertex cover} of 
{$G$ if every edge in $E$ is incident with a vertex in $A$.}
The size of the smallest vertex cover of $G$ is denoted by $\tau (G)$.
   
We will use the following two results to prove  \cref{thm:main}. The first shows that every vertex cover of a critically $k$-biconnected bipartite graph $G$ is relatively large. 
The second uses this result to show that if $k$ is sufficiently large, then $G$ has a $\cB_d(G)$-seed which does not cover its edges.

\begin{theorem}\label{lem:tau_bound}
    Let $G=(V,E)$ be a critically $k$-biconnected bipartite graph.
    Then $\tau(G)\geq \frac{|V|}{2k^2}$.
\end{theorem}

Since the proof of \cref{lem:tau_bound} uses different ideas than the rest of the paper, we defer it to the next section.

\begin{lemma}\label{lem:existsk}
    Let $d\geq 2$ and $k=10^5d^3$. Then every critically $k$-biconnected bipartite graph $G=(V,E)$ has a $\cB_d(G)$-seed $K$ with $|K|< \tau (G)$.
\end{lemma}
\begin{proof}
    Suppose that $G$ is a critically $k$-biconnected graph.
    Let $A$ be a vertex cover of $G$ of size $\tau(G)$, and define $A'=\{v\in A:|N_G(v)\cap A|\leq k-d\}.$
 Let $$p=\frac{d-1}{4d^3}\;\;\;\;\;\;  \text{ and }\;\;\; \;\;\; \eta=p\cdot \tau(G)+\frac{dn}{\exp(30d)}.$$ We claim the following:\medskip
 
\noindent \textbf{Claim.}    There exists a set $X_0\subseteq A$ of size $|X_0|\leq \eta$
        such that for every $v\in V-A'$, we have $|N_G(v)\cap X_0|\geq d$.\medskip
 
   We first show how this claim implies the statement of the lemma.
    Set $X_1=X_0\cup (V-A')$ and $X_2=V$. Note that $|N_G(v)\cap X_{i-1}|\geq d$ for all  $i\in \{1,2\}$ and $v\in X_i-X_{i-1}$.
    Thus, it follows from \cref{lem:small_kernel} that $G$ has a  $\cB_d(G)$-seed $K$ with
    \[ |K|< \eta\cdot \frac{2d^3}{d-1}= \frac{\tau(G)}{2}+\frac{2d^4n}{(d-1)\exp({30d})}< \frac{\tau(G)}{2}+\frac{n}{4k^2}. \]
    Using \cref{lem:tau_bound}, we get $|K|<\tau(G)$, as required.

    To prove our claim, we use a probabilistic argument.
    Let $S$ be a random subset of $A$ obtained by putting each vertex $v\in A$ into $S$, independently,  with probability $p$. Then $\E|S|=p\cdot \tau(G)$. Let $B_{S}=\{v\in V-A': |N_G(v)\cap S|< d\}$. For every $v\in V-A'$, we have $|N_G(v)\cap S|\sim {\mathrm{Bin}}(|N_G(v)\cap A|,p)$. Hence, applying  \cref{chernoff}(a) with $\alpha = 1/2$ gives
    $$ \Prob\big(v\in B_S)
    \leq \Prob\Big(|N_G(v)\cap S|<\frac{|N_G(v)\cap A|\cdot p}{2}\Big)\leq \exp\Big(\frac{-(k-d)p}{8}\Big).$$
        Since $\frac{(k-d)p}{8}\geq 30d$, it follows that $$\E|B_{S}|=\sum_{v\in V-A'}\Prob(v\in B_{S})\leq  n\exp({-30d}).$$
        We obtain
        $$\E\big(|S|+d|B_{S}|\big)=\E|S|+d\cdot\E|B_{S}|\leq \eta.$$
        Thus, with positive probability, $|S|+d|B_{S}|\leq \eta$. 
         We complete the proof of the claim by fixing such a set $S$, and letting $X_0$ be any set obtained by adding $d$ vertices from $N_G(v) \cap A$ to $S$ for each $v\in B_{S}$.
\end{proof}

\cref{thm:main}  follows immediately from our next result, using the fact that every $(2k-1)$-connected bipartite graph is $k$-biconnected.

\begin{theorem}\label{thm:bicon} For every integer $d\geq 1$, there exists an integer $k_d = O(d^3)$ such that every $k_d$-biconnected bipartite graph  is \birigid{d}.
\end{theorem}
\begin{proof}

If $d=1$, then $\cB_d(G)$ is the graphic matroid of $G$, and thus $G$ is $1$-birigid if and only if it is connected, which is further equivalent to $G$ being $1$-biconnected. Hence we may take $k_1 = 1$.

Therefore, let us assume that $d \geq 2$.
We prove that $ k = k_d = 10^5d^3$ suffices in this case. Assume, for a contradiction, that this is not true; that is, that there is a $k$-biconnected bipartite graph which is not \birigid{d}.
    Choose a counterexample $G=(V,E)$ such that $|V|$ is as small as possible and, subject to this condition,  $|E|$ is as large as possible. Let $(A,B)$ be the bipartition of $G$ and let $a=|A|$ and $b = |B|$. We may assume that $a \leq b$. If $a=k$, then the $k$-biconnectivity of $G$ implies that $G=K_{a,b}$ 
    and $G$ is \birigid{d}, a contradiction. Hence $a,b\geq k+1$.
    
    Suppose that $G-v$ is $k$-biconnected for some vertex $v\in V$. Then, by the minimality of $|V|$, $G-v$ is \birigid{d}.  Since we can obtain a spanning subgraph of $G$ from $G-v$ by a 0-extension operation, \cref{lem:ext}(a) implies that $G$ is also \birigid{d}, a contradiction.
    Hence no such vertex $v$ exists, that is, $G$ is critically $k$-biconnected. 

By \cref{lem:existsk}, we may choose  
    a $\cB_{d}(G)$-seed $K$ for $G$ such that $|K|< \tau (G)$. Then $K$ is also a $\cB_{d}(K_{a,b})$-seed of $G$ since the definition of a $\cB_{d}(K_{a,b})$-seed depends only on the restriction of $\cB_{d}(K_{a,b})$ to $E$. By \cref{lemma:deletablevertices,lem:neighbours:clique}(a), there exist $u,v\in V-K$ with $uv\in E$ such that
    for all $x\in N_G(u)-\{v\}$ and  $y\in N_G(v)-\{u\}$, the pair $\{x,y\}$ is $\cB_{d}(K_{a,b})$-linked in $G$. The maximality of $|E|$ now implies that $xy\in E$ for all such $x$ and $y$.
    Since $\min\{|V_1|,|V_2|\}\geq k+1$, the $k$-biconnectivity of $G$ implies that either $u$ or $v$ has degree at least $k+1$ in $G$.
    By symmetry, we may assume that $\deg_G(u)\geq k+1$. The previous observation then implies that to separate the neighbourhood of $v$ in $G$, we must delete at least $k+1$ vertices from the color class of $v$. It follows that $G-v$ is also $k$-biconnected, contradicting the fact that $G$ is critically $k$-biconnected.
\end{proof}

We believe that the statement of \cref{thm:main} remains true when $k_d = 2d^2$. If so, this bound would be best possible; see \cref{conjecture:abrigid} below.

\section{Vertex covers of critically \texorpdfstring{\boldmath $k$}{k}-biconnected graphs}

In order to finish our proof of \cref{thm:main}, it remains to prove \cref{lem:tau_bound}. We shall use the following result from the theory of graph connectivity. For a pair $u,v\in V$ in a graph $G=(V,E)$, let $\kappa(u,v;G)$ denote the maximum number of pairwise internally disjoint paths from $u$ to $v$ in $G$.

\begin{theorem}\cite[Theorem 2.16]{CT}\label{sparse}
    Let $G=(V,E)$ be a graph on $n$ vertices and let $k$ be a positive integer. Then $G$ has a spanning subgraph $H$ with $|E(H)|\leq kn - \binom{k+1}{2}$ such that \[\kappa(u,v;H)\geq \min \{k, \kappa(u,v;G) \}\] holds for all $u,v\in V$.
\end{theorem}

We say that the subgraph $H$ in \cref{sparse} is
a {\em sparse local certificate} of $G$ {\em with respect to $k$}.

Let $G=(V,E)$ be a $k$-biconnected bipartite graph with bipartition $(A,B)$. A separator $S=A'\cup B'$ of $G$ with $A'\subseteq A$ and $B'\subseteq B$ is called {\em essential} if $|A'|=k$ and $|B'|\leq k-1$ or $|B'|=k$ and $|A'|\leq k-1$ holds, and each vertex in $S$ has a neighbour in each connected component of $G-S$. Let $X\subseteq V$ and let \[\hat X=\{ x\in X : \hbox{there is an essential separator $S$ of $G$ with $x\in S$}\}.\] A function $f:\hat X\to V\times V$ is said to be a \emph{pairing} (for $X$) if, for each $x \in \hat X$, we have $f:x\mapsto (u,v)$ where $u,v$ are two neighbours of $x$ chosen from different components of $G-S$ for some essential separator $S$ of $G$ with $x\in S$. The pairing $f$ gives rise to a multigraph $G^f_X=(V,E^f_X)$ on vertex set $V$ and edge set $E^f_X=\{uv:(u,v)=f(x) \mbox{ for some } x\in \hat X\}$.

\begin{lemma}\label{cl}
    Let $G$, $\hat X$, $f$ and $G^f_X$ be as above. Then $G^f_X$ has $|\hat X|$ edges. In addition,
$E\cap E^f_X=\varnothing$
    and
    the multiplicity of each edge
$uv$ in $G^f_X$ is at most $k$.
\end{lemma}

\begin{proof}
    The first assertion follows from the definition of $G^f_X$.
    To see the second part, let $uv$ be an edge of $G^f_X$
    defined by some $x\in \hat X$ and some essential separator $S$ with $x\in S$.
    We may suppose that $u,v\in A$. Since $S$ separates $u$ and $v$ in $G$, $uv\not\in E$ and hence $E\cap E^f_X=\varnothing$.
    Also note that if $f(x') = (u,v)$ holds for some $x' \in \hat X$, then we necessarily have $x' \in B$ (since $G$ is bipartite) and $x' \in S$ (since $ux',vx' \in E$ and $S$ separates $u$ and $v$).
    Hence the multiplicity of $uv$ in $G^f_X$ is at most $|S\cap B|\leq k$, as claimed.
\end{proof}

\paragraph{Proof of \texorpdfstring{\cref{lem:tau_bound}}{Theorem 5.1}.}  \hspace{-1em}
Let $(A,B)$ be the bipartition of $V$. If $\min \{|A|,|B|\}=k$, then $G=K_{k,k}$ and $\tau(G)=k=\frac{|V|}{2}\geq \frac{|V|}{2k^2}$ holds. Hence we may assume that $|A|,|B|\geq k+1$.

Let $T\subseteq V$ be a smallest vertex cover of $G$, and let $X=V-T$. Let $\hat X$ be defined as above. Since $G$ is critically $k$-biconnected and $|A|,|B|\geq k+1$, for each vertex $x\in X$ (indeed, for each vertex $x \in V$) there exists an essential separator $S$ of $G$ with $x\in S$. Thus $\hat X=X$. Choose a pairing $f : X \to V \times V$ for $X$ and consider the multigraph $G^f_X$. Since $T$ is a vertex cover of $G$ and $T\cap X=\varnothing$, each edge of $G^f_X$ is induced by $T$. Let $F$ be obtained from $E^f_X$ by keeping only one copy of each edge, and let $G^+=G[T]\cup F$. By \cref{cl}, $G^+$ is a simple graph and $|X|\leq k|F|$.

\begin{claim}\label{cl2}
    For each $uv\in F$, we have $\kappa(u,v;G^+)\leq 2k-1$.
\end{claim}

\begin{proof}
    Fix $x \in X$ with $f(x)=(u,v)$. By definition, there is an essential separator $S$ with $x\in S\cap X$ in $G$ that separates $u$ and $v$. Observe that if $f(x')=(u',v')$ for some pair $u',v'$ separated by $S$, then $x'\in S\cap X$ must hold. Let $F'\subseteq F$ be the set of edges in $F$ whose end-vertices are separated by $S$ in $G$. Then the pair $u,v$ belongs to different components of $G^+-(S-X)-F'$, and hence $\kappa(u,v;G^+)\leq |S-X|+|F'|\leq |S-X|+|S\cap X|=|S|\leq 2k-1$.
\end{proof}

To complete the proof, we choose a sparse local certificate $H^+=(T,E^+)$ of $G^+$ with respect to $2k-1$, which exists by \cref{sparse}. \cref{cl2} implies that we have $\kappa(u,v;H^+) = \kappa(u,v;G^+)$ for every $uv \in F$, and hence $F\subseteq E^+$. Therefore \[|X|\leq k|F|\leq k|E^+| \leq k(2k-1)|T|,\] which gives $|V|-|T|\leq k(2k-1)|T|$ and $|V|\leq 2k^2|T|$. Hence $\tau(G)=|T|\geq \frac{|V|}{2k^2}$, as required. 
\qed
\medskip

\cref{lem:tau_bound} gives a lower bound on the size of vertex covers of critically $k$-biconnected bipartite graphs. We may adapt the proof of \cref{lem:tau_bound} to obtain a better bound on $\tau(G)$ when $G$ is \emph{critically $k$-connected}, i.e., when $G$ is $k$-connected  but $G - v$ is not $k$-connected for all vertices $v$ of  $G$. 

\begin{theorem}\label{lem:tau_bound2}
    Let $G=(V,E)$ be a critically $k$-connected graph. Then $\tau(G)\geq \frac{|V|}{k+1}.$
\end{theorem}

\begin{proof}
We may assume that $k\geq 2$. First suppose that $|V|\leq 3k-2$.
Then $\frac{|V|}{k+1}<3$, so the theorem follows unless $k\leq \tau(G)\leq 2$.
Moreover, $k=\tau(G)$ holds only if $G$ contains $K_{k,|V|-k}$ as a spanning subgraph. Now the criticality of $G$ and $k=2$ imply that $G$ is 
either $K_3$ or $C_4$, for which the statement is clear.

Thus we may assume that $|V|\geq 3k-1$. The rest of the proof and 
our notation is similar to
that of \cref{lem:tau_bound}, except that 
we shall call a separator $S$ of $G$ 
{\itshape essential} if $|S|=k$ holds. 
Since $G$ is critically $k$-connected, for each vertex
$x\in V$ there exists an essential separator  $S$ with $x\in S$. 
Let $T\subseteq V$ be a smallest vertex cover and 
put $X=V-T$.

Next we define a pairing $f:X\to V\times V$ for $X$.
We choose a minimal sequence of essential separators $S_1,S_2,\dots,S_r$ such that
$X\subseteq \bigcup_1^r S$, and for each $j$, from $j=1$ to $j=r$,
we define $f(x)$ for
each $x\in X\cap (S_j-\bigcup_{i=1}^{j-1} S_i)$ sequentially as follows.
Fix $S_j$.
Since $S_j$
is a $k$-separator and $|V|\geq 3k-1$, there is a component $C$ of $G-S_j$
such that $|V(G-S_j-V(C))|\geq k.$ Let $D=V-S_j-V(C)$.
Let $G'$ be obtained from $G$ by
adding a new vertex $p$ and $k$ edges from $p$ to different vertices of $D$.
Let $q\in V(C)$. Now $G'$ is $k$-connected, and hence there exist $k$
internally disjoint paths from $p$ to $q$ in $G'$ by Menger's theorem.
Each vertex $x\in S_j-\bigcup_{i=1}^{j-1} S_i$ belongs to a subpath of length two in this collection
of $k$ paths, with end-vertices $u\in V(C)$ and $v\in D$.
Then we let $f(x)=(u,v)$.

Thus $f$ defines a graph $G_X^f=(V,E_X^f)$.
Since $T$ is a vertex cover of $G$ and $T\cap X=\emptyset$, each edge of $G_X^f$ is induced by $T$.
The key observation is that $G_X^f$ is simple.
The construction of
the pairing implies that the pairs defined for a fixed $S_j$ are
pairwise different. 
Suppose that for $y\in S_j$ we have $f(y)=(u,v)=f(x)$ for some
$x\in \bigcup_1^{j-1} S_i$, where $f(x)$ was defined earlier for some
$S_i$. The vertex $y$ is connected to both $u$ and $v$, and $u,v$ are separated by $S_i$,
which implies that $y\in S_i$, a contradiction. 

We may now use the argument in \cref{cl} and  \cref{cl2} to deduce that $G^+=G[T]\cup E_X^f$ is simple and,
for each $uv\in E_X^f$, we have $\kappa(u,v;G^+)\leq k$.
Then a similar count to that in the proof of \cref{lem:tau_bound} gives $\tau(G)\geq \frac{|V|}{k+1}$.
\end{proof}

The following example shows that the bounds in \cref{lem:tau_bound,lem:tau_bound2} cannot be replaced by $\frac{|V|}{c}$ for any $c < k$. In particular, the bound $\frac{|V|}{k+1}$ in \cref{lem:tau_bound2} is almost tight.
Let $p$ be a positive integer, and let $A_i,B_i, i \in\{1,\ldots,p\}$ be disjoint sets of size $k$. Let us fix elements $a_i \in A_i$ and $b_i \in B_i, i \in \{1,\ldots,p\}$. Let $H_i$ denote the complete bipartite graph on bipartition $(A_i,B_i)$, and let $G$ be obtained from $H_1,\ldots,H_p$ by identifying the vertex sets $B_1 - b_1,\ldots,B_p - b_p$, as well as the vertices $a_1,\ldots,a_p$. It is not difficult to see that the resulting graph $G$ is critically $k$-connected, and in fact, critically $k$-biconnected. We have $|V(G)|=k(p+1)$ and $\tau(G)\leq k+p=\frac{|V(G)|}{k}+(k-1)$. Hence for any $c < k$, we can achieve $\tau(G) < \frac{|V(G)|}{c}$ by choosing a sufficiently large number $p$.

\section{The proof of \texorpdfstring{\cref{thm:redundant}}{Lemma 1.5}(b)}\label{sec:red}

We will need the following result from \cite{JJT2} which characterises the (global) birigidity of a graph in terms of the (global) completability of a related semisimple graph.

\begin{theorem}\label{theorem:birigidtocompletable}\cite[Theorems 17 and 19]{JJT2}
    Let $G$ be a bipartite graph with bipartition $(A,B)$, where $|A|,|B| \geq d$, and let $S \subseteq A$ be a set of $d$ vertices. Let $H$ be the semisimple graph obtained from $G$ by adding a looped complete graph on $S$. Then
    \begin{enumerate}
        \item $G$ is $d$-birigid if and only if $H$ is $d$-completable, 
        \item $G$ is globally $d$-birigid if and only if $H$ is globally $d$-completable.
    \end{enumerate}
\end{theorem}

We will also need the following result on global completablity  which follows immediately from  \cite[Theorem 22]{JJT2}.

\begin{theorem}\label{theorem:globaltolocal}
    Let $H$ be a semisimple graph, let $uv$ be a non-loop edge with $|N_H(u) \setminus \{u\}| \geq d+1$ and $|N_H(v) \setminus \{v\}| \geq d+1$, and let $H'$ be obtained by adding all edges between $N_H(u) \setminus \{u\}$ and $N_H(v) \setminus \{v\}$ (including loops at the vertices in $(N_H(u) \setminus \{u\}) \cap (N_H(v) \setminus \{v\})$). Suppose that $H-u$ and $H-v$ are $d$-completable and $H'$ is globally $d$-completable. Then $H$ is globally $d$-completable.
\end{theorem}

We can combine \cref{theorem:birigidtocompletable,theorem:globaltolocal} to obtain an analogue of \cref{theorem:globaltolocal} for global birigidity.

\begin{lemma}
    \label{lem:glob_birigid_v2}
    Let $G$ be a bipartite graph and $uv$ be an edge in $G$.  
    Suppose that $G-u$ and $G-v$ are both $d$-birigid, and that the graph $G'$ obtained by adding the set of edges 
    $\{u_iv_j : u_i \in N_G(v),\, v_j \in N_G(u)\}$  to $G$ is globally $d$-birigid. Then $G$ is globally $d$-birigid.
 \end{lemma}
\begin{proof}
If $|N_G(u)|\leq d$, then the hypothesis that $G-v$ is $d$-birigid implies that $G-v$ is a complete bipartite graph. We can now use the hypothesis that $G-u$ is birigid to deduce that either $|N_G(v)|\geq d+1$,  or $G$ is a complete bipartite graph. Since $G$ is globally $d$-birigid in both cases, we may assume that $|N_G(u)|\geq d+1$  and, by symmetry, $|N_G(v)|\geq d+1$. 

Let $(A,B)$ be the bipartition of $G$ with $u\in A$. Choose $S\subseteq A-u$ with $|S|=d$, and let $H,H'$ be the semisimple graphs obtained by adding the edges of a looped complete graph on $S$ to $G,G'$, respectively. Then 
\cref{theorem:birigidtocompletable} implies that $H-u,H-v$ are both $d$-completable, and $H'$ is globally $d$-completable. We can now use \cref{theorem:globaltolocal} to deduce that $H$ is globally $d$-completable, and one more application of 
\cref{theorem:birigidtocompletable}(b) tells us  that $G$  is globally $d$-birigid.
\end{proof}

\paragraph*{Proof of \cref{thm:redundant}(b).}  \hspace{-1em}
    Suppose, for contradiction, that $G$ is not globally $d$-birigid. We may assume that $G$ has the maximum number of edges among all counterexamples  with vertex set $V$. Then $G$ is a connected, non-complete bipartite graph, and hence there exist $u,v,u',v'\in V$ with $uv, uv', u'v\in E$ and $u'v'\notin E$. The maximality of $E$ implies that $G+\{u_iv_j : u_i \in N_G(v),\, v_j \in N_G(u)\}$ is globally $d$-birigid, and we can now use
     \cref{lem:glob_birigid_v2} to deduce that $G$ is globally $d$-birigid. This contradicts the choice of $G$.
\qed

\section{Concluding remarks}\label{section:concluding}

\subsection{Completability and hyperconnectivity of highly connected graphs}

We saw in the Introduction that there exist graphs of arbitrarily high connectivity which are not  \locallycompletable{d} or $d$-hyperconnected.  It is possible, however, that the following extension of \cref{thm:main} is true.

\begin{conjecture}\label{conjecture:completabilityLovaszYemini}
    For every positive integer $d$, there exists a positive integer $k_d$ such that every $k_d$-connected graph $G$ on $n$ vertices satisfies $\rank \cS_d(G)  \geq dn - d^2$ and $\rank \cH_d(G)  \geq dn - d^2$.
\end{conjecture}

We note that \cref{conjecture:completabilityLovaszYemini} would follow from \cref{thm:main} and a long-standing conjecture of Thomassen~\cite{thomassen_1989} that, for every positive integer $k$, every sufficiently highly connected graph contains a $k$-connected bipartite spanning subgraph.

\subsection{\texorpdfstring{\boldmath $(a,b)$}{(a,b)}-birigidity}\label{subsection:concludingbipartite}

Another generalisation of \cref{thm:main} concerns {$(a,b)$-birigidity}. 
This notion was introduced by Kalai, Nevo and Novik~\cite{KNN}, motivated in part by potential applications to upper bound conjectures for simplicial complexes and lower bound conjectures for cubic complexes. %
Given a pair of positive integers $a,b$ and a bipartite graph $G = (V,E)$ with bipartition $V = (X,Y)$, we define an \emph{$(a,b)$-realisation} of $G$ as a pair $(p,q)$, where $p : X \to \R^a$ and $q : Y \to \R^b$. The {\em  birigidity matroid of} $(G,p,q)$, denoted by $\cB_d(G,p,q)$, is the row matroid of the $|E|\times (b|X|+a|Y|)$ matrix $B(G,p,q)$ with rows indexed by $E$ and columns indexed by $(\{1,\ldots,b\} \times X) \cup (\{1,\ldots,a\} \times Y)$, in which the row indexed by an edge $xy\in E$ is 
\vspace{-.5em}
\[
\kbordermatrix{
& &  x & & y & \\
e=xy & 0 \dots 0 & q(y) & 0\dots 0 & p(x) & 0\dots 0
}.
\]
The {\em $(a,b)$-birigidity matroid of $G$}, $\cB_{a,b}(G)$, is given by $\cB(G,p,q)$ for any generic $(p,q)$. It is known that
\[
\rank \cB_{a,b}(K_{m,n})=\begin{cases}
    bm+an-ab \mbox{ if $m\geq a$ and $n\geq b$},\\  
    nm \mbox{  otherwise}.
\end{cases}
\]
We say that $G$ is \emph{$(a,b)$-birigid} if either $|X| \geq a$, $|Y| \geq b$ and $\rank \cB_{a,b}(G) = b|X| + a|Y| - ab$ holds,
or if $G$ is a complete bipartite graph. Note that when $a=b=d$, we recover the notion of $d$-birigidity used throughout the paper.

It follows from~\cite[Lemma 3.12]{KNN} that if a bipartite graph is $d$-birigid, then it is also $(a,b)$-birigid for all $a,b \leq d$. Thus \cref{thm:main} immediately implies the following result.

\begin{theorem}\label{theorem:abrigid}
    For every pair of integers $a,b \geq 1$, there exists an integer $k_{a,b}$ such that every $k_{a,b}$-connected bipartite graph is $(a,b)$-birigid.
\end{theorem}

The bound on $k_{a,b}$ obtained from 
the proof of \cref{thm:main} is probably far from tight. We conjecture that the statement of \cref{theorem:abrigid} holds with $k_{a,b} = 2ab$.

\begin{conjecture}\label{conjecture:abrigid}
    Every $2ab$-connected bipartite graph is $(a,b)$-birigid.
\end{conjecture}

We can modify a well-known example of Lovász and Yemini~\cite{LY} to show that the connectivity hypothesis of \cref{conjecture:abrigid} would be best possible when $(a,b)\neq (1,1)$. 
Let $k = 2ab - 1$, and 
let $G_0 = (V_0,E_0)$ be a $k$-connected, $k$-regular bipartite graph with bipartition $(X_0,Y_0)$, where $|X_0|=|Y_0| = s$ is an even integer with $s > k \geq ab$.
Let $G = (V,E)$ be the graph obtained from $G_0$ by splitting every vertex $v \in V_0$ into a set $A_v$ of $k$ vertices of degree one and a set $B_v$ of $k$ isolated vertices, and then adding all edges between $A_v$ and $B_v$,

It is straightforward to check that $G$ is $k$-connected. For each $v \in V_0$, let $G_v$ denote the copy of $K_{k,k}$ induced by $A_v \cup B_v$ in $G$.
Since $|A_v| = |B_v| = k \geq \max(a,b)$, we have $r_{a,b}(G_v) = (a+b)k - ab$ 
for all $v \in V_0$, where $r_{a,b}$ is the rank function of $\cB_{a,b}(G)$. Now by writing $E = E_0 \cup \bigcup_{v \in V_0} E(G_v)$ and using the submodularity of $r_{a,b}$, we obtain
\begin{align*}
    r_{a,b}(G) \leq |E_0| + \sum_{v \in V_0} r_{a,b}(G_v) &= ks + 2s\big((a+b)k - ab\big) \\[-5mm] 
    &= a \cdot 2sk + b \cdot 2sk - s(2ab-k) \\ &= a|X| + b|Y| - s \\ &< a|X| + b|Y| - ab.
\end{align*}
Hence $G$ is not $(a,b)$-birigid.

We close by noting that when $\min\{a,b\} = 1$, the birigidity matroid $\cB_{a,b}(G)$ coincides with the \emph{$k$-plane matroid} of $G$ introduced by Whiteley \cite{Wscene}, where $k = \max\{a,b\}$. In this case, \cref{conjecture:abrigid} holds by a result of Berg and Jord\'an (\cite[Theorem 4]{BJ}).

\subsection{Global hyperconnectivity and skew-symmetric matrix completion}\label{sec:skew}

It is not obvious how to define the global $d$-hyperconnectivity of a graph and, to our knowledge, this concept has not yet appeared in the literature. We will give a definition for the special case when $d$ is even, and discuss how this concept relates to the completion of low-rank skew-symmetric matrices. 

Let $G=(V,E)$ be a graph. For all $p:V\to \R^{2d}$, let
$p_1$ (respectively, $p_2$) be the projection of $p$ onto its first (respectively, last) $d$ coordinates. For $u,v\in V$ put \[p(u)*p(v)=p_1(u)\cdot p_2(v)-p_2(u)\cdot p_1(v).\]   
We say that the framework $(G,p)$ is {\em globally hyperconnected} if every $2d$-dimensional framework $(G,q)$ which satisfies $p(u)*p(v)=q(u)*q(v)$ for all $uv\in E$ also satisfies the stronger condition that $p(u)*p(v)=q(u)*q(v)$ for all $u,v\in V$. 
The graph $G$ is said to be {\em globally $2d$-hyperconnected} if $(G,p)$ is globally hyperconnected for every generic realisation $p$ of $G$ in $\R^{2d}$.

The following proposition describes the connection between (global) hyperconnectivity and the completion of partially filled low-rank skew-symmetric matrices. It is based on the observation that, given a realisation $p$ of $K_n$ in $\R^{2d}$, we can construct a skew-symmetric $n\times n$  matrix $M_p=(m_{i,j})$ by putting $m_{i,j}=p(v_i)*p(v_j)$ for all $1\leq i,j\leq n$. We will say that $M_p$ is a {\em generic skew-symmetric matrix} whenever the map $p$ is generic.

\begin{proposition}
    Let $G$ be a graph on $n$ vertices and let $d$ be a positive integer.
    \begin{enumerate}
        \item $G$ is $\mathcal{H}_{2d}$-independent if and only if 
        every partially filled $n \times n$ skew-symmetric matrix with generic entries in cells corresponding to the edges of $G$   can be completed to a skew-symmetric complex matrix of rank at most $2d$.
        \item $G$ is globally $2d$-hyperconnected if and only if 
        every generic  skew-symmetric $n\times n$ real matrix of rank $2d$ is uniquely determined by the subset of entries that correspond to the edges of $G$.
    \end{enumerate}
\end{proposition}
\begin{proof}[Proof sketch]
    Part \emph{(a)} is a restatement of the fact that $\mathcal{H}_{2d}$ is the algebraic matroid of the variety of $n \times n$ skew-symmetric complex matrices of rank at most $2d$, see \cite[Proposition 3.1]{RS} and note that, although  the statement of \cite[Proposition 3.1]{RS} is for real skew symmetric matrices,  the same proof works over the complexes. Part \emph{(b)} follows from the fact that 
    the mapping $p \mapsto M_p$ is surjective onto the variety of $n \times n$ real skew-symmetric matrices of rank at most $2d$, and
    from the definition of global hyperconnectivity. 
\end{proof}

It as an open problem to decide whether a sufficient condition analogous to \cref{thm:redundant} holds for global hyperconnectivity. 

\begin{conjecture}\label{conj:globalhyperconnected}
    Suppose $d$ is an integer and $G = (V,E)$ is a graph. If $G-v$ is $2d$-hyperconnected for all $v \in V$, then $G$ is globally $2d$-hyperconnected.
\end{conjecture}

Combined with \cref{thm:mindegree}(a), \cref{conj:globalhyperconnected} would give an analogue of \cref{thm:glob}(a) for global hyperconnectivity and, as a consequence, an analogue of \cref{thm:matrix}(a) for skew-symmetric matrices.

\subsection{Complex global completability  and symmetric matrix completion}\label{sec:complex}
We noted in the Introduction that a partially filled $n\times n$ matrix  with generic entries is completable to a symmetric complex matrix of rank $d$ if and only if the subset of $E(K_n^o)$ defined by the positions of the given entries in is independent in $\cS_d(K_n^o)$, and that a generic positive semidefinite   $n\times n$ matrix of rank $d$ is uniquely determined by a subset of its entries if and only if the spanning subgraph of $E(K_n^o)$ defined by the positions of the fixed entries is globally $d$-completable. In this context, it is natural to ask when a ``generic'' symmetric (real or complex) matrix of rank $d$ is uniquely determined by a subset of its entries. The fact that every symmetric  $n\times n$ matrix $M$ of rank $d$ can be factored as $M=A^TA$ for some $d\times n$ complex matrix $A$ indicates that we should consider the complex version of this problem, and define  $M=A^TA$ to be a {\em generic symmetric matrix} if the entries in $A$ are algebraically independent over $\rat$ for some such factorisation. This, in turn, leads us to consider frameworks in $\complex^d$. 

Given a graph $G=(V,E)$ and a realisation $p:V\to \complex^d$, we say that the framework $(G,p)$
 is {\em globally completable in $\complex^d$} if every realisation  $q:V\to  \complex^d$ which satisfies $q(u)^Tq(v)=p(u)^Tp(v)$ for all $uv\in E$ must also satisfy $q(u)^Tq(v)=p(u)^Tp(v)$ for all $u,v\in V$. The graph $G$ is said to be {\em globally completable in $\complex^d$} if every generic realisation of $G$ in $\complex^d$ is globally completable.
Since $\complex$ is algebraically closed, it is equivalent to say $G$ is {globally completable in $\complex^d$} if {\em some} generic realisation of $G$ in $\complex^d$ is globally completable.
(This statement does not hold for real realisations: an example of a graph which has two generic realisations in $\real^2$, one of which is globally completable   in $\real^2$ and the other is not, is given in \cite{JJT2}.)

The above definitions immediately imply that a generic symmetric $n\times n$ matrix of rank $d$ is uniquely determined by a subset $S$ of its  entries (in the set of all  symmetric $n\times n$ complex matrices of rank $d$) if and only if the spanning subgraph of $K_n^o$ defined by the positions of the entries in $S$ is globally completable in $\complex^d$.

 It seems likely that the sufficient condition for global completability in $\real^d$ given in \cref{thm:redundant}(a) extends to  $\complex^d$. If so, then this would give an analogous extension of Theorems \ref{thm:matrix}(a) and \ref{thm:glob}(a), and would imply, in particular, that every generic real symmetric matrix of rank $d$ 
is uniquely determined by any subset of its entries which includes at least $(n + d + 1)/2$ entries from each row.

\subsection{Matrix completion from a randomly chosen  set of entries}
We can use \cref{thm:glob}(b) to obtain a sharp threshold 
		for the number of randomly chosen entries required to uniquely determine a generic	
	 rank $d$, $m\times n$ matrix. Specifically, we consider the commonly studied model where the entries are chosen uniformly at random, i.e., the underlying graph is an Erdős--Rényi bipartite graph.
		By a result of Ruciński~\cite{Ruc}, for any integer $r \ge 1$ and constant $0 < \alpha \le 1$, a random bipartite graph 
        with parts of sizes $m$ and $n = \alpha m$ and 
        $m\log m + r\log\log m$ edges is asymptotically almost surely $r$-connected.
		When $r = k_d+1$, \cref{thm:glob}(b) implies that such graphs are globally $d$-birigid. Equivalently:
		\begin{corollary}\label{coro:random:matrix}
			Let $\alpha$ be a constant with $0< \alpha\leq 1$. The probability that a generic matrix in $M_d(m,\alpha m)$ is uniquely determined by $m\log m+(k_d+1)\log\log m$ entries chosen uniformly at random tends to $1$ as $m$ tends to infinity.
		\end{corollary}
		The bound is sharp up to the factor of $\log \log m$, as $m\log m$ edges are necessary already for connectivity. An alternative proof of this statement can be obtained by following the approach of Hamaguchi and Tanigawa \cite{HT2024} (which is partly based on \cite{LNPR}), who recently studied the
        problem of higher-order tensor completion with randomly sampled entries. 
        
        It is known that if the 
        matrix $M$ in $M_d(m,\alpha m)$
        satisfies a so-called {\emph{incoherence} property (or is drawn from a corresponding random model), then $\Omega(dm \log m)$ randomly chosen entries are necessary to uniquely determine $M$; moreover, there exists an efficient
        algorithm that recovers $M$ from $O(dm \log m)$ random entries \cite{CT2010}. 
        However, these bounds do not directly carry over to our setting,
        as genericity and incoherence are not comparable properties,  see  \cite{KTT}.
        In fact, \cref{coro:random:matrix} shows that if we consider generic matrices, then the information-theoretic bound on the number of entries asymptotically equals $m\log m$ for any fixed $d$.
        We should emphasise, however,  that
        our results
        do not provide an algorithm to recover the matrix $M$ from the given set of entries.

\section*{Acknowledgements}

BJ was supported by the MTA Distinguished Guest Scientist Fellowship Programme 2025.
TJ was supported by the National Research, Development and Innovation Office
of Hungary, grant no. Advanced 152786, 
the MTA-ELTE Momentum Matroid Optimization Research Group, and the National Research, Development and Innovation Fund of Hungary, financed under the ELTE TKP 2021-NKTA-62 funding scheme.

\printbibliography

\end{document}